\documentclass[12pt,a4paper]{article}
\usepackage[utf8]{inputenc}
\usepackage{amsmath, amssymb, amsthm, verbatim, hyperref}
\usepackage{mathrsfs}
\usepackage[dvipsnames]{xcolor}
\usepackage{ragged2e}
\usepackage{marginnote}
\usepackage{enumerate}
\usepackage{makecell}
\usepackage{longtable}
\usepackage{epsfig}
\usepackage{tikz}
\usepackage{ytableau}
\usepackage{hyperref}
\usepackage{caption}
\usepackage{subcaption}
\hypersetup{hidelinks}
\usepackage{fancyvrb}
\usepackage{tikz-cd}
\usepackage[left=2.50cm, right=2.50cm, top=2.00cm, bottom=2.00cm]{geometry}
\usetikzlibrary {arrows.meta}
\usepackage{ifthen}

\newcommand{\cross}[3]{
    \draw[thick] (-0.5+#1, -#2) to (0.5+#1, -#2+1);
    \draw[thick] (-0.5+#1, -#2+1) to (0.5+#1, -#2);
    \foreach \j in {0,...,#3}{
    \ifthenelse{\equal{\j}{\the\numexpr #2-1} \OR \j = #2}{}{
    \draw[thick] (-0.5+#1,-\j ) to (0.5+#1,-\j);}
    }
}

\newcommand{\uncross}[3]{
    \draw[thick] [bend right = 90, looseness=1.25] (-0.5+#1, -#2) to (-0.5+#1, -#2+1);
    \draw[thick] [bend right = 90, looseness=1.25] (0.5+#1, -#2+1) to (0.5+#1, -#2);
    \foreach \j in {0,...,#3}{
    \ifthenelse{\equal{\j}{\the\numexpr #2-1} \OR \j = #2}{}{
    \draw[thick] (-0.5+#1,-\j ) to (0.5+#1,-\j);}
    }
}

\def\nodd{[n]_{\text{odd}}}
\def\Cov{\textsf{Cov}}
\def\Imm{\operatorname{Imm}}
\def\TL{\operatorname{TL}}

\definecolor{goodgreen}{rgb}{0.01, 0.75, 0.24}

\theoremstyle{plain}

\newtheorem{thm}{Theorem}[section]
\newtheorem{prop}[thm]{Proposition}
\newtheorem{cor}[thm]{Corollary}
\newtheorem{lemma}[thm]{Lemma}
\newtheorem{conjecture}[thm]{Conjecture}

\theoremstyle{definition}
\newtheorem{definition}[thm]{Definition}
\newtheorem{example}[thm]{Example}
\theoremstyle{remark}

\newtheorem{remark}[thm]{Remark}

\newcommand{\ZZ}{\mathbb{Z}}

\newcommand{\CC}{\mathbb{C}}

\newcommand{\NN}{\mathbb{N}}

\title{\vspace{-1em} 
Temperley--Lieb Immanants of Ribbon Decomposition Matrices}
\author{\vspace{-2em} \\
Son Nguyen, Pavlo Pylyavskyy}
\date{\vspace{-3em}}

\allowdisplaybreaks

\begin{document}
\ytableausetup{centertableaux}

\maketitle

    \begin{abstract}
        Ribbon decomposition matrices give determinantal formulas for skew Schur functions that include as special cases the classical Jacobi–Trudi, Giambelli, and Lascoux–Pragacz formulas.
We prove that certain elements of Lusztig’s dual canonical basis, called Temperley–Lieb immanants, are Schur-positive when evaluated on ribbon decomposition matrices.
We conjecture that this positivity holds for all elements of the dual canonical basis. This is known in the special case of Jacobi–Trudi matrices by a result of Haiman.
    \end{abstract}

\tableofcontents


\section{Introduction}

    Since Lusztig's groundbreaking 1990 paper \cite{lusztig1990canonical}, \textit{canonical bases} and \textit{dual canonical bases} have inspired several major developments in mathematics, such as Kashiwara crystal bases \cite{kashiwara1991crystal, kashiwara1993global} and Fomin--Zelevinsky cluster algebras \cite{fomin2002cluster, berenstein1996parametrizations}. In type A, dual canonical bases are certain polynomials in the matrix algebra whose coefficients are given as evaluations of Kazhdan--Lusztig polynomials, see \cite{du1992canonical, du1995canonical,rhoades2006kazhdan, skandera2008dual}. For example, the determinant of the whole matrix is an element of dual canonical basis.

    In type A, dual canonical bases have a deep Schur positivity property. It is well-known that determinant of a Jacobi--Trudi matrix is a skew Schur function and hence is Schur positive. It turns out that, like determinant, all elements of dual canonical bases evaluate to Schur positive results on generalized Jacobi--Trudi matrices. This was proved by Rhoades--Skandera in \cite[Proposition 3]{rhoades2006kazhdan}, relying on Haiman's result \cite[Theorem 1.5]{haiman1993hecke}, which in turn relies on the (proofs of) Kazhdan--Lusztig conjecture by Beilinson--Bernstein \cite{BB} and Brylinski--Kashiwara \cite{brylinski1981kazhdan}. No direct combinatorial proof of this property is known. However, for a subset of dual canonical bases, called \textit{Temperley--Lieb immanants}, introduced by Rhoades--Skandera in \cite{rhoades2005temperley}, Nguyen--Pylyavskyy \cite{nguyen2025temperley} gave a combinatorial proof using \textit{shuffle tableaux}.

    Besides the Jacobi--Trudi identity, many determinantal formulas for skew Schur functions are known. Most well-known examples include the dual Jacobi--Trudi identity, Giambelli formula, and Lascoux--Pragacz formula \cite{lascoux1988ribbon}. Here, we first introduce \textit{ribbon decomposition matrices} (see Definition \ref{def:ribbon-mat}), which specialize to all examples above. Using the technology of shuffle tableaux, we prove that their Temperley--Lieb immannants are also Schur positive.

    \begin{thm}\label{thm:main-thm}
        Given an infinite ribbon $R$, a skew shape $\lambda/\mu$ compatible with $R$. All Temperley--Lieb immannants of the ribbon decomposition matrix $A_{\lambda/\mu,R}$ are Schur positive.
    \end{thm}

    We conjecture that all Kazhdan--Lusztig immannants of ribbon decomposition matrices are also Schur positive.

    \begin{conjecture}\label{conj:KL}
        Given an infinite ribbon $R$, a skew shape $\lambda/\mu$ compatible with $R$. All Kazhdan--Lusztig immannants of the ribbon decomposition matrix $A_{\lambda/\mu,R}$ are Schur positive.
    \end{conjecture}

    It seems to us that proving Conjecture \ref{conj:KL} would require even deeper tools than Haiman's proof for Jacobi--Trudi matrices.

    \begin{remark}
        Our ribbon decomposition matrices are special cases of Hamel--Goulden matrices \cite{hamel1995planar}. It is natural to ask whether Kazhdan--Lusztig immanants of Hamel--Goulden matrices are Schur positive. However, it turns out that some Hamel--Goulden matrices do not have Schur positive Temperley--Lieb immanants. For example, for the following matrix from \cite[Figure 6]{hamel1995planar}
        \[
        \left(\begin{matrix}
                \scalebox{0.5}{
                \begin{ytableau}
                    {} \\
                    {}
                \end{ytableau}} & 
                \scalebox{0.5}{
                \begin{ytableau}
                     \none & \\
                     & \\
                \end{ytableau}} & 
                \scalebox{0.5}{
                \begin{ytableau}
                     {}
                \end{ytableau}} & 
                \scalebox{0.5}{
                \begin{ytableau}
                     \none & \none & \\
                      & & \\
                      {}
                \end{ytableau}} \\[0.5cm]
                \scalebox{0.5}{
                \begin{ytableau}
                    {} & {} & {} \\
                    {} \\
                    {} 
                \end{ytableau}} & 
                \scalebox{0.5}{
                \begin{ytableau}
                    \none & & & \\
                     \none & \\
                     {} & {}
                \end{ytableau}} & 
                \scalebox{0.5}{
                \begin{ytableau}
                    {} & {} & {} \\
                    {}
                \end{ytableau}} & 
                \scalebox{0.5}{
                \begin{ytableau}
                    \none & \none & & & \\
                    \none & \none & \\
                     {} & {} & {} \\
                     {}
                \end{ytableau}} \\
                \scalebox{0.5}{
                \begin{ytableau}
                    {} \\
                    {} \\
                    {}
                \end{ytableau}} & 
                \scalebox{0.5}{
                \begin{ytableau}
                    \none & \\
                    \none & \\
                     & \\
                \end{ytableau}} & 
                \scalebox{0.5}{
                \begin{ytableau}
                    {} \\
                    {}
                \end{ytableau}} & 
                \scalebox{0.5}{
                \begin{ytableau}
                    \none & \none & \\
                    \none & \none & \\
                     {} & {} & {} \\
                     {}
                \end{ytableau}} \\[0.7cm]
                1 & 
                \scalebox{0.5}{
                \begin{ytableau}
                    
                     {}
                \end{ytableau}} & 
                0 & 
                \scalebox{0.5}{
                \begin{ytableau}
                    {} & {} \\
                    {}
                \end{ytableau}}
            \end{matrix} \right),
        \]
        its Kazhdan--Lusztig immanant $\Imm_{2143}$, which is also a Temperley--Lieb immanant, is not Schur positive.
    \end{remark}

\section{Ribbon decomposition}\label{sec:ribbon}

\subsection{Skew shapes and ribbons}

    A \textit{partition} is a sequence $\lambda = (\lambda_1,\lambda_2,\ldots,\lambda_\ell)$ of weakly decreasing nonnegative integers. The \textit{size} of $\lambda$ is $|\lambda| = \lambda_1 + \cdots + \lambda_\ell$. The \textit{Young diagram} of $\lambda$ is an array of left-justified square cells, where the $i$-th row from top to bottom contains $\lambda_i$ cells. We index a box in row $i$ column $j$ as $(i,j)$. For partitions $\lambda,\mu$, we say $\mu \subseteq \lambda$ is $\mu_i\leq \lambda_i$ for all $i$. The \textit{skew shape} $\lambda/\mu$ is the set theoretic difference $\lambda - \mu$ of their Young diagrams.

    A \textit{semistandard Young tableau} (SSYT) $T$ of shape $\lambda/\mu$ is a filling of its skew Young diagram with positive integers such that the numbers are weakly increasing along each row and strictly increasing along each column. The weight of a SSYT T, denoted $\operatorname{wt}(T)$, is $\alpha = (\alpha_1,\alpha_2,\ldots)$, where $\alpha_i$ is the number of $i$'s in $T$. The \textit{skew Schur function} is defined to be
    \[ s_{\lambda/\mu} = \sum_{\text{SSYT $T$ of shape $\lambda/\mu$}}x^{\operatorname{wt}(T)}, \]
    where $x^{\alpha} = x_1^{\alpha_1}x_2^{\alpha_1}\cdots$. When $\mu = \emptyset$, the function $s_{\lambda}$ is a \textit{Schur function}. Schur functions form a basis for the ring of symmetric function. A symmetric function is \textit{Schur positive} if it is a positive sum of Schur functions.

    A \textit{ribbon} is a connected skew shape with no $2\times 2$ block. An \textit{infinite ribbon} is a ribbon with extra $1\times\infty$ or $\infty\times 1$ blocks at the beginning and the end. Let $R$ be an infinite ribbon, we can tile the plane with infinitely many copies of $R$. A skew shape $\lambda/\mu$ is \textit{compatible} with $R$ is the intersection of $\lambda/\mu$ with each copy of $R$ is contiguous.

    \begin{example}\label{ex:ribbon-compatible}

        Figure \ref{subfig:ribbon} shows an infinite ribbon $R$ with an $\infty\times 1$ block at the beginning and a $1 \times\infty$ block at the end. Figure \ref{subfig:ribbon-tiling} shows the tiling by infinitely many copies of $R$.

        \begin{figure}[h!]
         \centering
            \begin{subfigure}[c]{0.45\textwidth}
                \centering
                \includegraphics[scale = 1]{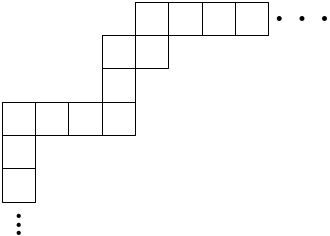}
                \caption{An infinite ribbon}
                \label{subfig:ribbon}
            \end{subfigure}
         \quad
            \begin{subfigure}[c]{0.45\textwidth}
                \centering
                \includegraphics[scale = 1]{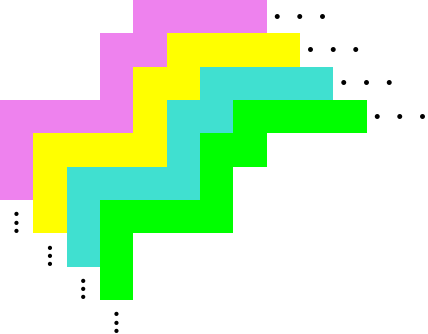}
                \caption{Tiling}
                \label{subfig:ribbon-tiling}
            \end{subfigure}
    
            \caption{}
            \label{fig:ribbon}
        \end{figure}

        Figure \ref{subfig:ribbon-compatible} shows a skew shape that is compatible with $R$. The skew shape in Figure \ref{subfig:ribbon-not-compatible}, on the other hand, is not compatible since the intersection with the yellow copy is not contiguous.

        \begin{figure}[h!]
         \centering
            \begin{subfigure}[c]{0.45\textwidth}
                \centering
                \includegraphics[scale = 1]{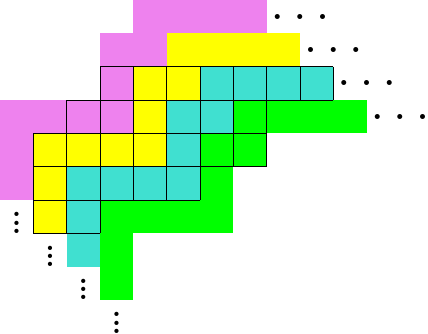}
                \caption{Example}
                \label{subfig:ribbon-compatible}
            \end{subfigure}
         \quad
            \begin{subfigure}[c]{0.45\textwidth}
                \centering
                \includegraphics[scale = 1]{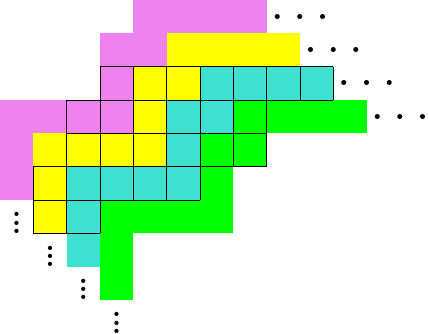}
                \caption{Non-example}
                \label{subfig:ribbon-not-compatible}
            \end{subfigure}
    
            \caption{}
            \label{fig:ribbon-examples}
        \end{figure}
    \end{example}

\subsection{Ribbon decomposition matrices}

    Now we define ribbon decomposition matrices. They are special cases of Hamel--Goulden matrices in \cite{hamel1995planar}.

    Given a square $(i,j)\in \ZZ^2$, its \textit{content} is $j-i$. Here, the $i$-coordinate increases downward while the $j$-coordinate increases rightward. Every infinite ribbon $R$ has exactly one box of each content. Thus, we can denote $[a,b]_R$, for $a,b\in \NN$, to be the section of the ribbon between content $a$ and $b$ inclusive. We denote $(a,b)_R,[a,b)_R,(a,b]_R$ similarly.

    Given an infinite ribbon $R$ and a skew shape $\lambda/\mu$ compatible with $R$. From inside out, let $\theta_1,\ldots,\theta_\ell$ be the intersections of $\lambda/\mu$ and copies of $R$. Each $\theta_i$ is an interval $[a_i,b_i)_R$ for some $a_i < b_i$. Hence, we can identify $\lambda/\mu$ with two tuples $\Bar{a} = (a_1,a_2,\ldots,a_\ell)$ and $\Bar{b} = (b_1,b_2,\ldots,b_\ell)$. We call $\Bar{a}$ the ending tuple and $\Bar{b}$ the starting tuple of $\lambda/\mu$.

    \begin{definition}\label{def:ribbon-mat}
        Given an infinite ribbon $R$ and a skew shape $\lambda/\mu$ compatible with $R$. The $\Bar{a}$ and $\Bar{b}$ be the ending and starting tuples of $\lambda/\mu$. The ribbon decomposition matrix of $\lambda/\mu$ with respect to $R$ is
        \[ A_{\lambda/\mu,R} = \left(s_{[a_j,b_i)_R}\right)_{i,j=1}^n, \]
        with the convention $s_{[a_j,b_i)_R} = 1$ if $a_j = b_i$ and $0$ is $a_j > b_i$.
    \end{definition}

    \begin{example}\label{ex:ribbon-tuple}
        Figure \ref{fig:ribbon-tuple} shows a skew shape with contents of its squares. We have $\Bar{a} = (0,-4,-3,3)$ and $\Bar{b} = (3,5,9,6)$.
    
        \begin{figure}[h!]
            \centering
            \includegraphics[scale = 1]{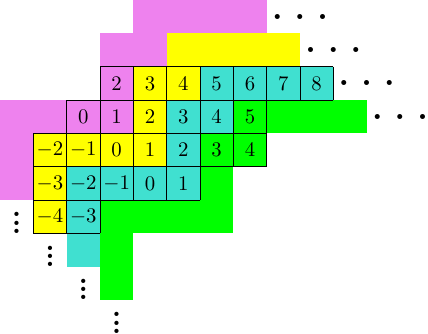}
            \caption{Skew shape and content}
            \label{fig:ribbon-tuple}
        \end{figure}

        The corresponding ribbon decomposition matrix is
        \[ \left(\begin{matrix}
            \scalebox{0.5}{
            \begin{ytableau}
                \none & \\
                 & 
            \end{ytableau}} & 
            \scalebox{0.5}{
            \begin{ytableau}
                 \none & \none & \none & \\
                 & & & \\
                 {} \\
                 {}
            \end{ytableau}} & 
            \scalebox{0.5}{
            \begin{ytableau}
                 \none & \none & \none & \\
                 & & & \\
                 {}
            \end{ytableau}} & 
            1 \\
            \scalebox{0.5}{
            \begin{ytableau}
                \none & & \\
                \none & \\
                 & 
            \end{ytableau}} & 
            \scalebox{0.5}{
            \begin{ytableau}
                \none & \none & \none & & \\
                 \none & \none & \none & \\
                 & & & \\
                 {} \\
                 {}
            \end{ytableau}} & 
            \scalebox{0.5}{
            \begin{ytableau}
                \none & \none & \none & & \\
                 \none & \none & \none & \\
                 & & & \\
                 {}
            \end{ytableau}} & 
            \scalebox{0.5}{
            \begin{ytableau}
                {} & {}
            \end{ytableau}} \\
            \scalebox{0.5}{
            \begin{ytableau}
                \none & \none & & & & \\
                \none & & \\
                \none & \\
                 & 
            \end{ytableau}} & 
            \scalebox{0.5}{
            \begin{ytableau}
                \none & \none & \none & \none & & & & \\
                \none & \none & \none & & \\
                 \none & \none & \none & \\
                 & & & \\
                 {} \\
                 {}
            \end{ytableau}} & 
            \scalebox{0.5}{
            \begin{ytableau}
                \none & \none & \none & \none & & & & \\
                \none & \none & \none & & \\
                 \none & \none & \none & \\
                 & & & \\
                 {}
            \end{ytableau}} & 
            \scalebox{0.5}{
            \begin{ytableau}
                 \none & & & & \\
                {} & {}
            \end{ytableau}} \\
            \scalebox{0.5}{
            \begin{ytableau}
                \none & \none & \\
                \none & & \\
                \none & \\
                 & 
            \end{ytableau}} & 
            \scalebox{0.5}{
            \begin{ytableau}
                \none & \none & \none & \none & \\
                \none & \none & \none & & \\
                 \none & \none & \none & \\
                 & & & \\
                 {} \\
                 {}
            \end{ytableau}} & 
            \scalebox{0.5}{
            \begin{ytableau}
                \none & \none & \none & \none & \\
                \none & \none & \none & & \\
                 \none & \none & \none & \\
                 & & & \\
                 {}
            \end{ytableau}} & 
            \scalebox{0.5}{
            \begin{ytableau}
                \none & \\
                {} & {}
            \end{ytableau}}
        \end{matrix} \right) \]
    \end{example}

    Ribbon decomposition matrices are special cases of Hamel--Goulden matrices. Thus, their determinants are skew Schur functions.

    \begin{thm}[{\cite[Theorem 3.1]{hamel1995planar}}]\label{thm:hamel-goulden}
        For any ribbon $R$ and skew shape $\lambda/\mu$ compatible with $R$, we have
        \[ \det(A_{\lambda/\mu, R}) = s_{\lambda/\mu}. \]
    \end{thm}

    \begin{proof}[Proof sketch]
        The result follows directly from a Lindstr\"om–-Gessel–-Viennot argument similar to the Jacobi--Trudi identity. The network will be introduced in Section \ref{subsec:network}.
    \end{proof}

    For some choices of the ribbon $R$, ribbon decomposition matrices are well-known matrices whose determinants are (skew) Schur functions. 
    
    \begin{enumerate}
       \item When $R$ is an infinite $1\times \infty$ row, every skew shape $\lambda/\mu$ is compatible with $R$, and $A_{\lambda/\mu, R}$ is the \textit{Jacobi--Trudi matrix}.
       \item When $R$ is an infinite $\infty\times 1$ column, $A_{\lambda/\mu, R}$ is the \textit{dual Jacobi--Trudi matrix}, or \textit{N\"{a}egelsbach--Kostka matrix}. 
       \item When $R$ is an infinite hook, any straight shape $\lambda$ is compatible with $R$, and $A_{\lambda,R}$ is the \textit{Giambelli matrix}.
       \item When $R$ coincides with the outer strip of $\lambda$, $A_{\lambda,R}$ is the \textit{Lascoux--Pragacz matrix} in \cite{lascoux1988ribbon}.
    \end{enumerate}

    Ribbon decomposition matrices are closed under taking principal minors, a property that does not hold for general Hamel-Goulden matrices.

    \begin{prop}\label{prop:minors}
        Suppose $\lambda/\mu$ has ending and starting tuples $\Bar{a},\Bar{b}$ of length $n$. Let $I\subseteq [n]$, then the $I\times I$ minor of $A_{\lambda/\mu, R}$ is $A_{\lambda'/\mu', R}$, where the ending and starting tuples of $\lambda'/\mu'$ are $(a_i~|~i\in I)$ and $(b_i~|~i\in I)$.
    \end{prop}

    \begin{proof}
        It suffices to prove the statement for $|I| = n-1$, and suppose $r = [n] -I$. We want to show that if we contract the ribbon $\theta_r$, the remaining boxes still form a skew shape, then the ending and starting tuples statement is immediate. Let $D$ be the diagram after contracting $\theta_r$. If $D$ is not a skew shape, then there is squares $(i,j)$, $(i',j')$, and $(i'', j'')$ such that $i\leq i'' \leq i'$, $j\leq j'' \leq j'$ and $(i,j), (i',j')\in D$ but $(i'',j'')\notin D$. If $\theta_r$ is originally NW of $(i,j)$, then adding back $\theta_r$, $(i,j)$ becomes $(i+1,j+1)$,  $(i',j')$ becomes $(i'+1,j'+1)$, and $(i'',j'')$ becomes $(i''+1,j''+1)$ (see Figure \ref{fig:ribbon-proof-1}). Then, $(i+1,j+1), (i'+1,j'+1)\in \lambda/\mu$ but $(i''+1,j''+1)\notin\lambda/\mu$, contradiction.

        \begin{figure}[h!]
         \centering
            \begin{subfigure}[c]{0.45\textwidth}
                \centering
                \includegraphics[scale = 1]{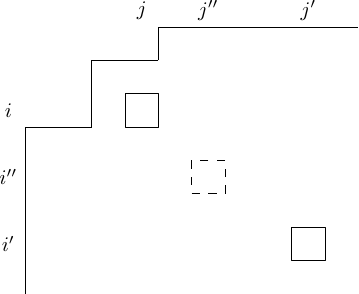}
                \caption{After contracting $\theta_r$}
                \label{subfig:ribbon-proof-1}
            \end{subfigure}\quad
            \begin{subfigure}[c]{0.45\textwidth}
                \centering
                \includegraphics[scale = 1]{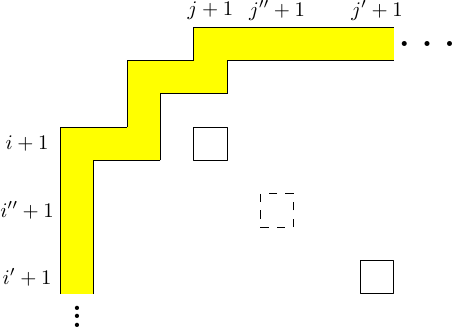}
                \caption{Original}
                \label{subfig:ribbon-proof-1b}
            \end{subfigure}
            \caption{}
            \label{fig:ribbon-proof-1}
        \end{figure}

        Similarly, if $\theta_r$ is originally SE of $(i,j)$ but NW of $(i'', j'')$, then $(i,j), (i'+1,j'+1)\in \lambda/\mu$ but $(i''+1,j''+1)\notin\lambda/\mu$. If $\theta_r$ is originally SE of $(i'',j'')$ but NW of $(i', j')$, then $(i,j), (i'+1,j'+1)\in \lambda/\mu$ but $(i'',j'')\notin\lambda/\mu$. If $\theta_r$ is originally SE of $(i',j')$, then $(i,j), (i',j')\in \lambda/\mu$ but $(i'',j'')\notin\lambda/\mu$. All of the above leads to a contradiction.
    \end{proof}

    \begin{example}
        Continuing Example \ref{ex:ribbon-tuple}, contract the yellow ribbon, we get the skew shape in Figure \ref{fig:ribbon-contract}. The ribbon decomposition matrix for this skew shape is the $\{1,3,4\}\times \{1,3,4\}$ minor of the matrix in Example \ref{ex:ribbon-tuple}.
    \end{example}
    
    \begin{figure}[h!]
        \centering
        \begin{minipage}{.5\textwidth}
            \centering
            \includegraphics[scale = 1]{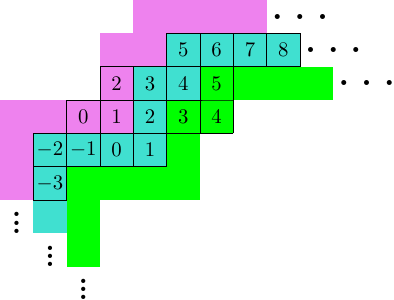}
            \caption{}
            \label{fig:ribbon-contract}
        \end{minipage}%
        \begin{minipage}{.5\textwidth}
            \centering
            \includegraphics[scale = 1]{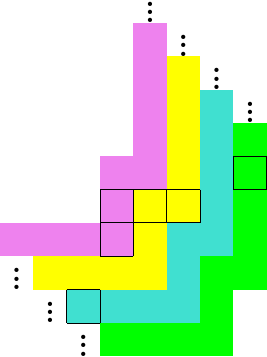}
            \caption{}
            \label{fig:badMinor}
        \end{minipage}
    \end{figure}

    \begin{remark}
        It is natural to ask if all minors of ribbon decomposition matrices are also Schur positive. The answer is no. For example, the ribbon decomposition matrix for the ribbon and skew shape in Figure \ref{fig:badMinor} is
        \[ 
        \left(\begin{matrix}
            \scalebox{0.5}{
            \begin{ytableau}
                {} \\
                {}
            \end{ytableau}} & 
            1 & 
            \scalebox{0.5}{
            \begin{ytableau}
                 \none & \none & \none & \\
                 & & & 
            \end{ytableau}} & 
            0 \\[0.5cm]
            \scalebox{0.5}{
            \begin{ytableau}
                {} & {} \\
                {} \\
                {} 
            \end{ytableau}} & 
            \scalebox{0.5}{
            \begin{ytableau}
                 {} & {}
            \end{ytableau}} & 
            \scalebox{0.5}{
            \begin{ytableau}
                \none & \none & \none & & \\
                \none & \none & \none & \\
                 & & & 
            \end{ytableau}} & 
            0 \\[0.5cm]
            0 & 
            0 & 
            \scalebox{0.5}{
            \begin{ytableau}
                 {}
            \end{ytableau}} & 
            0 \\[0.2cm]
            \scalebox{0.5}{
            \begin{ytableau}
                \none & {} \\
                \none & {} \\
                \none & {} \\
                {} & {} \\
                {} \\
                {} 
            \end{ytableau}} & 
            \scalebox{0.5}{
            \begin{ytableau}
                \none & {} \\
                \none & {} \\
                \none & {} \\
                {} & {}
            \end{ytableau}} & 
            \scalebox{0.5}{
            \begin{ytableau}
                \none & \none & \none & \none & \\
                \none & \none & \none & \none & \\
                \none & \none & \none & \none & \\
                \none & \none & \none & & \\
                \none & \none & \none & \\
                 & & & 
            \end{ytableau}} & 
            \scalebox{0.5}{
            \begin{ytableau}
                {}
            \end{ytableau}}
        \end{matrix} \right). 
        \]
        The minor of rows $1,2,4$ and columns $1,2,3$ is not Schur positive. However, Theorems \ref{thm:main-thm} and \ref{thm:prod-minor} together imply that any product of two complementary minors is Schur positive. We conjecture that any product of $k$ complementary minors is also Schur positive.
    \end{remark}

\subsection{Networks and noncrossing paths}\label{subsec:network}

    The proof Theorem \ref{thm:hamel-goulden} is an application of Lindstr\"om–Gessel–Viennot lemma. We need to create a network with starting points $P_1,\ldots,P_n$ and ending points $Q_1,\ldots,Q_n$ such that families of noncrossing paths from $P_i$'s to $Q_i$'s correspond to SSYTs of shape $\lambda/\mu$.

    \begin{definition}
        Given an infinite ribbon $R$, recall that $R$ has a box $r_i$ for each content $i\in \ZZ$. We define the network $\mathcal{N}_R$ as follows. The vertices of $\mathcal{N}_R$ are $(i,j)$, where $i\in \ZZ$ and $j\in \NN_{>0}$. There are four types of directed edges
        \begin{enumerate}
            \item \textit{diagonal edges} $(i+1,j-1)\rightarrow (i,j)$, for all $j\geq 2$ and $i$ such that box $r_{i-1}$ is below box $r_{i}$,
            \item \textit{horizontal edges} $(i+1,j)(i,\rightarrow j)$, for all $j\geq 1$ and $i$ such that box $r_{i-1}$ is left of box $r_{i}$,
            \item \textit{up edges} $(i,j) \rightarrow (i,j+1)$, for all $j\geq 1$ and $i$ such that box $r_{i-1}$ is below box $r_{i}$,
            \item \textit{down edges} $(i,j+1) \rightarrow (i,j)$, for all $j\geq 1$ and $i$ such that box $r_{i-1}$ is left of box $r_{i}$.
        \end{enumerate}

        Given a skew shape $\lambda/\mu$ compatible with $R$, and let $\Bar{a}$ and $\Bar{b}$ be the ending and starting tuples. We put the starting point $P_i$ at $(b_i,1)$ if box $r_{b_i-1}$ is below box $r_{b_i}$ and at $(b_i,\infty)$ otherwise. We put the ending point $Q_i$ at $(a_i,1)$ if box $r_{a_i-1}$ is left of box $r_{a_i}$ and at $(a_i,\infty)$ otherwise. We denoted $\mathcal{N}_{\lambda/\mu,R}$ the network $\mathcal{N}_R$ with the starting and ending points $P_i$'s, $Q_i$'s from the skew shape $\lambda/\mu$.

        The weights of the edges of $\mathcal{N}_R$ are as follows. The up and down edges have weight $1$. The diagonal and horizontal edges have weight $x_j$ if the source vertex is $(i,j)$ for some $i$. We denote the weight of an edge $e$ by $\omega(e)$. The weight of a path $\pi_i$ from $P_i$ to $Q_i$ is defined to be
        \[ \omega(\pi_i) = \prod_{e\in\pi_i}\omega(e). \]
        The weight of a family of noncrossing paths is the product of the weights of its paths.
    \end{definition}

    It is easy to check that each family of noncrossing paths from $P_i$'s to $Q_i$'s in $\mathcal{N}_{\lambda/\mu,R}$ corresponds to a SSYT of shape $\lambda/\mu$: the horizontal and diagonal edges on the path from $P_i$ to $Q_i$ give the entries on the ribbon $\theta_i$. In similar fashion, each entry $s_{[a_j,b_i)}$ in $A_{\lambda/\mu,R}$ counts weights of paths from $P_i$ to $Q_j$.

    \begin{example}\label{ex:network}
        Continuing Example \ref{ex:ribbon-tuple}, Figure \ref{subfig:network} shows the network $\mathcal{N}_{\lambda/\mu,R}$ and a family of noncrossing paths. The starting points are marked by filled circles while the ending points are marked by hollow circles. Figure \ref{subfig:ribbon-SSYT} shows the corresponding SSYT.

        \begin{figure}[h!]
         \centering
            \begin{subfigure}[c]{\textwidth}
                \centering
                \includegraphics[scale = 0.7]{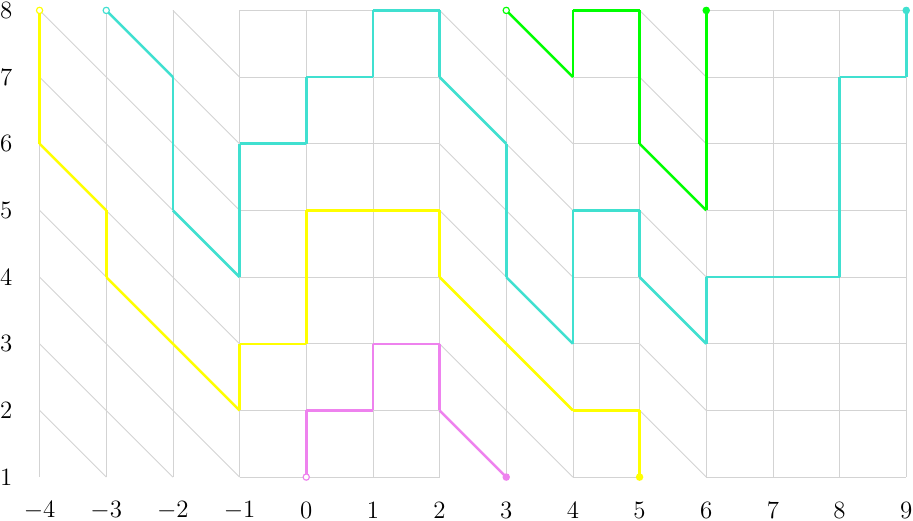}
                \caption{Noncrossing paths}
                \label{subfig:network}
            \end{subfigure}
         
         \vspace{2em}
            \begin{subfigure}[c]{\textwidth}
                \centering
                \includegraphics[scale = 1]{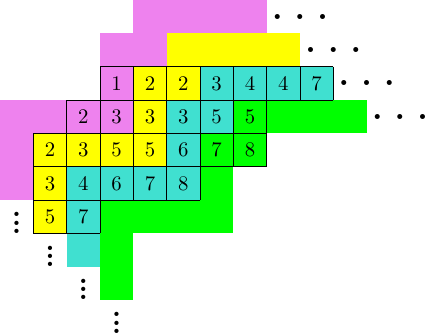}
                \caption{SSYT}
                \label{subfig:ribbon-SSYT}
            \end{subfigure}
    
            \caption{}
            \label{fig:network-SSYT}
        \end{figure}
    \end{example}

    To compute Temperley--Lieb immanants, we need more families of paths.

    \begin{definition}
        A \textit{(generalized) wiring} $H$ on a network $\mathcal{N}_{\lambda/\mu,R}$ is a family of $n$ paths $\pi_1,\ldots,\pi_n$ such that
        \begin{itemize}
            \item $\pi_i$ goes from $P_i$ to $Q_i$ for all $i$, and
            \item no three paths share a vertex.
        \end{itemize}
        The weight of a wiring is the product of the weights of its paths, i.e.
        \[ \omega(H) = \prod_{i = 0}^n\omega(\pi_i). \]
    \end{definition}

\section{Temperley--Lieb immanants and shuffle tableaux}

\subsection{Temperley--Lieb immanants}

    The \textit{Temperley--Lieb algebra} $TL_n(2)$ is the $\CC$-algebra generated by $t_1,\ldots,t_{n-1}$ subject to the relations
    \begin{align*}
        t_i^2 &= 2 t_i,&\hspace{-5em}&\text{for $i = 1,\ldots,n-1$}, \\
        t_it_jt_i &= t_i,&\hspace{-5em}&\text{if $|i-j| = 1$}, \\
        t_it_j &= t_jt_i,&\hspace{-5em}&\text{if $|i-j| > 1$}.
    \end{align*}

    A \textit{$321$-avoiding permutation} is a permutation in $S_n$ that has no reduced decomposition of the form $\cdots s_is_js_i\cdots$ where $|i-j| = 1$. Because of the second relation, $TL_n(2)$ has a natural basis $\{t_w~|~\text{$w$ is a 321-avoiding permutation}\}$, where $t_w := t_{i_1}\ldots t_{i_k}$ for a reduced decomposition $w = s_{i_1}\ldots s_{i_k}$. The map
    \[ \theta~:~s_i \rightarrow t_i - 1 \]
    determines a homomorphism $\theta~:~\CC[S_n] \rightarrow TL_n(2)$. For any $w\in S_n$ and any 321-avoiding permutation $u\in S_n$, let $f_u(w)$ be the coefficient of $t_u$ in $\theta(w) = (t_{i_1}-1)\ldots (t_{i_k}-1)$ for a reduced decomposition $w = s_{i_1}\ldots s_{i_k}$. In \cite{rhoades2005temperley}, Rhoades and Skandera defined the \textit{Temperley--Lieb immanant} of an $n\times n$ matrix $X = (x_{ij})$ by
    \[ \Imm_u^{\TL}(X) = \sum_{w\in S_n}f_u(w)x_{1,w(1)}\ldots x_{n,w(n)}. \]

    Products of $t_i$'s in $TL_n(\xi)$ can be computed graphically using \textit{Temperley--Lieb diagrams}. The generators $t_1,t_2,t_3,\ldots$ are represented by
    \[ \resizebox{.05\textwidth}{!}{
        \begin{tikzpicture}
            \uncross{1}{1}{3}
            \draw[thick] (1.5,-5) to (0.5, -5);
            \draw (1,-4) node[anchor=center, scale = 1.5] {$\vdots$};
        \end{tikzpicture}
    },\quad 
    \resizebox{.05\textwidth}{!}{
        \begin{tikzpicture}
            \uncross{1}{2}{3}
            \draw[thick] (1.5,-5) to (0.5, -5);
            \draw (1,-4) node[anchor=center, scale = 1.5] {$\vdots$};
        \end{tikzpicture}
    },\quad 
    \resizebox{.05\textwidth}{!}{
        \begin{tikzpicture}
            \uncross{1}{3}{3}
            \draw[thick] (1.5,-5) to (0.5, -5);
            \draw (1,-4) node[anchor=center, scale = 1.5] {$\vdots$};
        \end{tikzpicture}
    },\ldots\]
    and multiplication is represented by concatenation.
    
    \begin{example}
        The Temperley--Lieb diagram for $t_1t_3t_2t_1t_3\in TL_4$ is
        \[ \resizebox{!}{.15\textwidth}{
            \begin{tikzpicture}
                \draw (0,0) node[anchor=center] {$L_1$};
                \draw (0,-1) node[anchor=center] {$L_2$};
                \draw (0,-2) node[anchor=center] {$L_3$};
                \draw (0,-3) node[anchor=center] {$L_4$};
                \draw (6,0) node[anchor=center] {$R_1$};
                \draw (6,-1) node[anchor=center] {$R_2$};
                \draw (6,-2) node[anchor=center] {$R_3$};
                \draw (6,-3) node[anchor=center] {$R_4$};
                \filldraw[black] (0.5,0) circle (2pt);
                \filldraw[black] (0.5,-1) circle (2pt);
                \filldraw[black] (0.5,-2) circle (2pt);
                \filldraw[black] (0.5,-3) circle (2pt);
                \filldraw[black] (5.5,0) circle (2pt);
                \filldraw[black] (5.5,-1) circle (2pt);
                \filldraw[black] (5.5,-2) circle (2pt);
                \filldraw[black] (5.5,-3) circle (2pt);
                \uncross{1}{1}{3}
                \uncross{2}{3}{3}
                \uncross{3}{2}{3}
                \uncross{4}{1}{3}
                \uncross{5}{3}{3}
            \end{tikzpicture}
        } \]
    \end{example}

    Each diagram consists of a noncrossing matching of the boundary vertices and internal loops. If two Temperley--Lieb diagrams have the same matching and the same number of internal loops, then the corresponding products are equal to each other. If the diagram of $a$ is obtained from the diagram of $b$ by removing $k$ internal loops, then $b = \xi^ka$ in $TL_n$. Thus, we have a natural bijection between basis elements of $TL_n$ and noncrossing matchings of $2n$ vertices.

\subsection{Product of minors}

    For two subsets $I,J\in [n]$ such that $|I| = |J|$, we color the boundary vertices as follows. Color $L_i$ black if $i\in I$ and white otherwise. Color $R_j$ white if $j\in J$ and black otherwise. A Temperley--Lieb basis element is \textit{compatible} with $(I,J)$ if each strand in the diagram has one black endpoint and one white endpoint.

    Denote by $\Theta(I,J)$ the set of Temperley--Lieb basis elements that are compatible with $(I,J)$. Let $\Bar{I} := [n] \backslash I$, and let $\Delta_{I,J}(A)$ denote the minor of $A$ in the row set $I$ and column set $J$. We have the following identity.
    \begin{thm}[{\cite[Proposition 4.3]{rhoades2005temperley}}, cf. \cite{skandera2004inequalities}]\label{thm:prod-minor}
        For two subsets $I,J\in [n]$ such that $|I| = |J|$, we have
        \[ \Delta_{I,J}(A)\cdot\Delta_{\Bar{I},\Bar{J}}(A) = \sum_{\tau\in\Theta(I,J)}\Imm^{\TL}_\tau(A). \]
    \end{thm}

    From now on, we will use
    \[ I = \nodd =  \{i \in [n]~|~i~\text{is odd}\}. \]
    The reason for this choice of $I$ is the following observation.
    \begin{prop}[{\cite[Proposition 2.15]{nguyen2025temperley}}]\label{prop:compatible}
        Let $I = \nodd$, every type $\tau$ is compatible with $(I,I)$.
    \end{prop}
    \begin{cor}[{\cite[Corollary 2.16]{nguyen2025temperley}}]\label{cor:all-type}
        For $I = \nodd$,
        \[ \Delta_{I,I}(A)\cdot\Delta_{\Bar{I}\Bar{I}}(A) = \sum_{\tau}\Imm^{\TL}_\tau(A), \]
        where the sum is over all Temperley--Lieb basis elements.
    \end{cor}

    Recall from Proposition \ref{prop:minors} that principal minors of $A_{\lambda/\mu, R}$ are $A_{\lambda'/\mu',R}$ for some $\lambda'/\mu'$. Let $\Bar{a},\Bar{b}$ be the ending and starting tuples of $\lambda/\mu$. We define $\lambda_R/\mu_R$ such that its ending and starting tuples are $(a_i~|~i~\text{is odd})$ and $(b_i~|~i~\text{is odd})$. We define $\lambda_B/\mu_B$ such that its ending and starting tuples are $(a_i~|~i~\text{is even})$ and $(b_i~|~i~\text{is even})$. By Proposition \ref{prop:minors},
    \[ s_{\lambda_R/\mu_R} = \Delta_{I,I}(A_{\lambda/\mu,R}),~\text{and}~s_{\lambda_B/\mu_B} = \Delta_{\Bar{I},\Bar{I}}(A_{\lambda/\mu,R}). \]
    Then, by Corollary \ref{cor:all-type},
    \[ s_{\lambda_R/\mu_R}\cdot s_{\lambda_B/\mu_B} = \sum_{\tau}\Imm^{\TL}_\tau(A), \]
    where the sum is over all Temperley--Lieb basis elements. The analogue of this on the level of wirings is as follows.

    \begin{definition}\label{def:colored-cover}
        Given a wiring $H$, a colored cover of $H$ is a coloring of the edges in $H$ by two colors red and blue that forms two families of red and blue paths satisfying:
        \begin{enumerate}
            \item every edge in $H$ is colored; in particular, the doubly covered edges are colored by two colors;
            \item the red paths are noncrossing paths from $\{P_i~|~i\in I\}$ to $\{Q_i~|~i\in I\}$, and the blue paths are noncrossing paths from $\{P_i~|~i\in \Bar{I}\}$ to $\{Q_i~|~i\in \Bar{I}\}$.
        \end{enumerate}
        We denote $\omega(x) = \omega(H)$ the weight of $x$. Let $\Cov(\mathcal{N}_{\lambda/\mu,R})$ be the set of colored covers of all generalized wirings $H$ on $\mathcal{N}_{\lambda/\mu,R}$. Given $x\in\Cov(\mathcal{N}_{\lambda/\mu,R})$, we read the Temperley--Lieb diagram by \textit{uncrossing} the intersections as follows.
        \begin{enumerate}
            \item Contract any doubly covered subpath to a single vertex.
            \item If two sources coincide at the same vertex $v$, create two new sources, each having one edge going to $v$. If two sinks coincide at the same vertex $v$, create two new sinks, each having one edge coming from $v$.
            \item For each vertex $v$ of indegree two and outdegree two, create vertex $v'$ with indegree two and vertex $v''$ with outdegree two.
        \end{enumerate}
        We denote $\psi(x)$ the resulting noncrossing matching.
    \end{definition}

    \begin{lemma}[{\cite[Corollary 2.11]{nguyen2025temperley}}]
        For any basis element $\tau$ of $TL_n(2)$, we have
        \[ \Imm^{\TL}_\tau(A_{\lambda/\mu,R}) = \sum_{\substack{x\in\Cov(\mathcal{N}_{\lambda/\mu,R})\\\psi(x) = \tau}} \omega(x). \]
    \end{lemma}

    \begin{example}\label{ex:colored-network}

        Figure \ref{subfig:colored-network} shows a colored cover of the network in Example \ref{ex:network}. Figure \ref{subfig:colored-network-uncrossed} shows this colored cover with all intersections uncrossed and doubly covered edges contracted.

        \begin{figure}[h!]
         \centering
            \begin{subfigure}[c]{\textwidth}
                \centering
                \includegraphics[scale = 0.7]{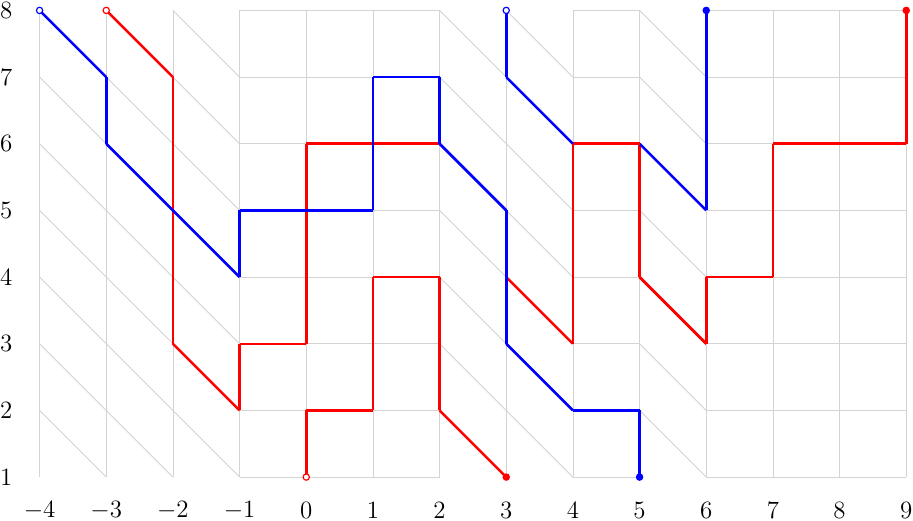}
                \caption{Colored cover}
                \label{subfig:colored-network}
            \end{subfigure}
         
         \vspace{2em}
            \begin{subfigure}[c]{\textwidth}
                \centering
                \includegraphics[scale = 0.7]{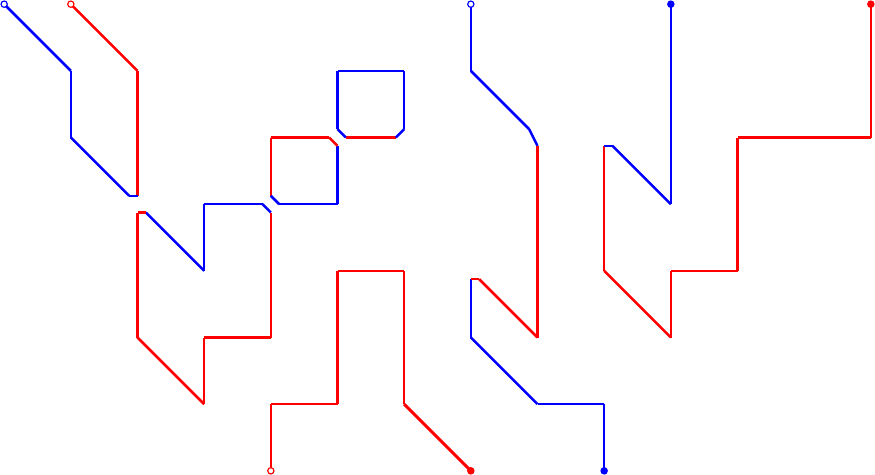}
                \caption{Colored cover uncrossed}
                \label{subfig:colored-network-uncrossed}
            \end{subfigure}
    
            \caption{}
            \label{fig:colored-network}
        \end{figure}

        Hence, the Temperley--Lieb type of this colored cover is
        \[ \resizebox{!}{0.2\textwidth}{
            \begin{tikzpicture}
                \draw (0,0) node[anchor=center] {$P_1$};
                \draw (0,-1) node[anchor=center] {$P_2$};
                \draw (0,-2) node[anchor=center] {$P_3$};
                \draw (0,-3) node[anchor=center] {$P_4$};
                \draw (3,0) node[anchor=center] {$Q_1$};
                \draw (3,-1) node[anchor=center] {$Q_2$};
                \draw (3,-2) node[anchor=center] {$Q_3$};
                \draw (3,-3) node[anchor=center] {$Q_4$};
                \filldraw[black] (0.5,0) circle (2pt);
                \filldraw[black] (0.5,-1) circle (2pt);
                \filldraw[black] (0.5,-2) circle (2pt);
                \filldraw[black] (0.5,-3) circle (2pt);
                \filldraw[black] (2.5,0) circle (2pt);
                \filldraw[black] (2.5,-1) circle (2pt);
                \filldraw[black] (2.5,-2) circle (2pt);
                \filldraw[black] (2.5,-3) circle (2pt);
                \uncross{1}{3}{3}
                \uncross{2}{2}{3}
            \end{tikzpicture}
        }. \]
    \end{example}

\subsection{Shuffle tableaux}

    Now we translate colored covers into shuffle tableaux, introduced by \cite{nguyen2025temperley}. Recall that we defined $\lambda_R/\mu_R$ and $\lambda_B/\mu_B$ such that 
    \[ s_{\lambda_R/\mu_R} = \Delta_{I,I}(A_{\lambda/\mu,R}),~\text{and}~s_{\lambda_B/\mu_B} = \Delta_{\Bar{I},\Bar{I}}(A_{\lambda/\mu,R}). \]
    For each colored cover in the previous section, the red paths together form a SSYT $T_R$ of shape $\lambda_R/\mu_B$, and the blue paths together form a SSYT $T_B$ of shape $\lambda_B/\mu_B$.

    \begin{example}\label{ex:SSYT-R-B}
        The colored cover in Example \ref{ex:colored-network} gives the following two SSYTs
        \[ T_R = \scalebox{1}{
            \textcolor{red}{\begin{ytableau}
                 \none & \none & \none & \none & 3 & 4 & 6 & 6 \\
                 \none & \none & 1 & 3 & 6 \\
                 \none & 2 & 4 & 5 \\
                 2 & 3 & 6 & 6 \\
                 7
            \end{ytableau}}}, \quad\quad\quad 
            T_B = \scalebox{1}{\textcolor{blue}{
            \begin{ytableau}
                 \none & \none & \none & 2 & 2 & 5 \\
                 \none & \none & \none & 5 & 6 & 6 \\
                 4 & 5 & 5 & 7 \\
                 5 \\
                 7
            \end{ytableau}}}. \]
    \end{example}

    \begin{definition}\label{def:shuffle}
        Given two skew shapes $\lambda_R/\mu_R$ and $\lambda_B/\mu_B$, the \textit{shuffle diagram} $D$ of shape $(\lambda_R/\mu_R)\circledast(\lambda_B/\mu_B)$ is constructed as follows.

        \begin{itemize}
            \item If $\lambda_R/\mu_R$ has a box in position $(i,j)$, then add into $D$ a box in position $(2i-1,2j-1)$.
            \item If $\lambda_B/\mu_B$ has a box in position $(i,j)$, then add into $D$ a box in position $(2i,2j)$.
        \end{itemize}
        
        A \textit{shuffle tableau} of shape $(\lambda_R/\mu_R)\circledast(\lambda_B/\mu_B)$ is a filling of the shuffle diagram of shape $(\lambda_R/\mu_R)\circledast(\lambda_B/\mu_B)$ such that the entries are weakly increasing along the rows and strictly increasing along the columns. The \textit{weight} of a shuffle tableau $T$, denoted $\omega(T)$, is $\prod_{i\geq 1} x_i^{\alpha_i}$, where $\alpha_i$ is the number of $i$-entries in $T$.
    \end{definition}

    In other words, a shuffle tableau of shape $(\lambda_R/\mu_R)\circledast(\lambda_B/\mu_B)$ is obtained by interlacing a SSYT of shape $\lambda_R/\mu_R$ and a SSYT of shape $\lambda_B/\mu_B$. Thus, there is a one-to-one correspondence from colored networks $x\in \Cov(\mathcal{N}_{\lambda/\mu,R})$ to shuffle tableaux of shape $(\lambda_R/\mu_R)\circledast(\lambda_B/\mu_B)$. Denote this map $\Phi$. Clearly, this map is weight-preserving, i.e. $\omega(x) = \omega(\Phi(x))$.

    \begin{example}\label{ex:shuffle-tableaux}
        Interlacing the two SSYTs $T_R$ and $T_B$ in Example \ref{ex:SSYT-R-B}, we get the following shuffle tableau
        \[ \begin{ytableau}
        \none & \none & \none & \none & \none & \none & \none & \none & \none & \textcolor{red}{3} & \none & \textcolor{red}{4} & \none & \textcolor{red}{6} & \none & \textcolor{red}{6} \\
        \none & \none & \none & \none & \none & \none & \textcolor{blue}{2} & \none & \textcolor{blue}{2} & \none & \textcolor{blue}{5} \\
        \none & \none & \none & \none & \none & \textcolor{red}{1} & \none & \textcolor{red}{3} & \none & \textcolor{red}{6} \\
        \none & \none & \none & \none & \none & \none & \textcolor{blue}{5} & \none & \textcolor{blue}{6} & \none & \textcolor{blue}{6} \\
        \none & \none & \none & \textcolor{red}{2} & \none & \textcolor{red}{4} & \none & \textcolor{red}{5} \\
        \textcolor{blue}{4} & \none & \textcolor{blue}{5} & \none & \textcolor{blue}{5} & \none & \textcolor{blue}{7} \\
        \none & \textcolor{red}{2} & \none & \textcolor{red}{3} & \none & \textcolor{red}{6} & \none & \textcolor{red}{6} \\
        \textcolor{blue}{5} \\
        \none & \textcolor{red}{7} \\
        \textcolor{blue}{7}
        \end{ytableau}. \]
    \end{example}

    To each shuffle tableaux $T$ we associate its {\it {Temperley--Lieb type}}, which is a non-crossing matching.  Recall that ribbons in the skew shapes $\lambda_R/\mu_R$ and $\lambda_B/\mu_B$ come from ribbons in the original skew shape $\lambda/\mu$. Hence to each ribbon in $\lambda/\mu$ we can associated the corresponding collection of cells in  $(\lambda_R/\mu_R)\circledast(\lambda_B/\mu_B)$.
    
    Next, we create the nodes that will serve as vertices of our non-crossing matching. Let $[a_k,b_k)$ be a ribbon in $\lambda/\mu$, and suppose the box $r_{a_k}$ of this ribbon corresponds to the box $(i,j)$ in $(\lambda_R/\mu_R)\circledast(\lambda_B/\mu_B)$. If box $r_{a_{k-1}}$ is below box $r_{a_k}$ in $R$, then we put $Q_k$ at coordinate $(i+1,j)$ (below $(i,j)$). Else, we put $Q_k$ at coordinate $(i,j-1)$ (left of $(i,j)$). Similarly, suppose the box $r_{b_{k-1}}$ of this ribbon corresponds to the box $(i',j')$ in $(\lambda_R/\mu_R)\circledast(\lambda_B/\mu_B)$. If box $r_{b_{k}}$ is above box $r_{b_{k-1}}$ in $R$, then we put $P_k$ at coordinate $(i'-1,j')$ (above $(i',j')$). Else, we put $P_k$ at coordinate $(i',j'+1)$ (right of $(i',j')$).

    Finally, for each pair
    \[ \begin{ytableau} 
    \none & j \\ 
    i & \none
    \end{ytableau}, \]
    if $i \leq j$, then draw two lines
    \[ \begin{tikzcd}[sep = tiny]
	{} & j \\
	i & {}
	\arrow[no head, from=2-1, to=1-1]
	\arrow[no head, from=1-2, to=2-2]
    \end{tikzcd}; \]
    if $i > j$, then draw two lines
    \[ \begin{tikzcd}[sep = tiny]
	{} & j \\
	i & {}
	\arrow[no head, from=2-1, to=2-2]
	\arrow[no head, from=1-2, to=1-1]
    \end{tikzcd}. \]
    These lines together form a noncrossing matching of the points $P_i$'s and $Q_i$'s. This is the desired Temperley--Lieb type of $T$, denoted $\psi(T)$.

    \begin{example}\label{ex:shuffle-to-TL}
        From the shuffle tableau in Example \ref{ex:shuffle-tableaux}, we get the following lines
        \[\begin{tikzcd}[sep=tiny]
        	&&&&&&&& {\textcolor{blue}{P_2}} & {\textcolor{red}{3}} & {} & {\textcolor{red}{4}} & {} & {\textcolor{red}{6}} & {} & {\textcolor{red}{6}} & {\textcolor{red}{P_3}} \\
        	&&&&& {\textcolor{red}{P_1}} & {\textcolor{blue}{2}} & {} & {\textcolor{blue}{2}} & {} & {\textcolor{blue}{5}} & {\textcolor{blue}{P_4}} \\
        	&&&&& {\textcolor{red}{1}} & {} & {\textcolor{red}{3}} & {} & {\textcolor{red}{6}} & {} \\
        	&&&&& {} & {\textcolor{blue}{5}} & {} & {\textcolor{blue}{6}} & {} & {\textcolor{blue}{6}} \\
        	&& {\textcolor{red}{Q_1}} & {\textcolor{red}{2}} & {} & {\textcolor{red}{4}} & {} & {\textcolor{red}{5}} & {\textcolor{blue}{Q_4}} \\
        	{\textcolor{blue}{4}} & {} & {\textcolor{blue}{5}} & {} & {\textcolor{blue}{5}} & {} & {\textcolor{blue}{7}} & {} \\
        	{} & {\textcolor{red}{2}} & {} & {\textcolor{red}{3}} & {} & {\textcolor{red}{6}} & {} & {\textcolor{red}{6}} \\
        	{\textcolor{blue}{5}} & {} \\
        	{} & {\textcolor{red}{7}} \\
        	{\textcolor{blue}{7}} & {\textcolor{red}{Q_3}} \\
        	{\textcolor{blue}{Q_2}}
        	\arrow[no head, from=1-10, to=1-11]
        	\arrow[no head, from=1-10, to=2-10]
        	\arrow[no head, from=1-11, to=1-12]
        	\arrow[no head, from=1-12, to=1-13]
        	\arrow[no head, from=1-14, to=1-13]
        	\arrow[no head, from=1-15, to=1-14]
        	\arrow[no head, from=1-16, to=1-15]
        	\arrow[no head, from=1-17, to=1-16]
        	\arrow[no head, from=2-6, to=3-6]
        	\arrow[no head, from=2-7, to=2-8]
        	\arrow[no head, from=2-7, to=3-7]
        	\arrow[no head, from=2-9, to=1-9]
        	\arrow[no head, from=2-9, to=2-8]
        	\arrow[no head, from=2-10, to=2-11]
        	\arrow[no head, from=2-11, to=2-12]
        	\arrow[no head, from=3-6, to=4-6]
        	\arrow[no head, from=3-8, to=3-7]
        	\arrow[no head, from=3-8, to=3-9]
        	\arrow[no head, from=3-9, to=4-9]
        	\arrow[no head, from=3-10, to=3-11]
        	\arrow[no head, from=3-10, to=4-10]
        	\arrow[no head, from=3-11, to=4-11]
        	\arrow[no head, from=4-6, to=5-6]
        	\arrow[no head, from=4-7, to=4-8]
        	\arrow[no head, from=4-7, to=5-7]
        	\arrow[no head, from=4-9, to=5-9]
        	\arrow[no head, from=4-10, to=4-11]
        	\arrow[no head, from=5-3, to=5-4]
        	\arrow[no head, from=5-4, to=5-5]
        	\arrow[no head, from=5-6, to=5-5]
        	\arrow[no head, from=5-8, to=4-8]
        	\arrow[no head, from=5-8, to=5-7]
        	\arrow[no head, from=6-1, to=6-2]
        	\arrow[no head, from=6-1, to=7-1]
        	\arrow[no head, from=6-3, to=6-4]
        	\arrow[no head, from=6-3, to=7-3]
        	\arrow[no head, from=6-4, to=7-4]
        	\arrow[no head, from=6-5, to=6-6]
        	\arrow[no head, from=6-5, to=7-5]
        	\arrow[no head, from=6-7, to=6-8]
        	\arrow[no head, from=6-7, to=7-7]
        	\arrow[no head, from=6-8, to=7-8]
        	\arrow[no head, from=7-1, to=7-2]
        	\arrow[no head, from=7-2, to=6-2]
        	\arrow[no head, from=7-4, to=7-3]
        	\arrow[no head, from=7-6, to=6-6]
        	\arrow[no head, from=7-6, to=7-5]
        	\arrow[no head, from=7-8, to=7-7]
        	\arrow[no head, from=8-1, to=8-2]
        	\arrow[no head, from=8-1, to=9-1]
        	\arrow[no head, from=8-2, to=9-2]
        	\arrow[no head, from=9-1, to=10-1]
        	\arrow[no head, from=9-2, to=10-2]
        	\arrow[no head, from=10-1, to=11-1]
        \end{tikzcd}.\]
        One can check that this gives the same Temperley--Lieb type as Example \ref{ex:colored-network}.
    \end{example}

    \begin{prop}
        Let $x$ be a colored cover of $\mathcal{N}_{\lambda/\mu,R}$, and let $T = \Phi(x)$ be the corresponding shuffle tableau. Then $x$ and $T$ have the same Temperley--Lieb type, i.e.
        \[ \psi(x) =  \psi(T). \]
    \end{prop}

    \begin{proof}
        The proof is verbatim to that of  \cite[Proposition 3.8]{nguyen2025temperley}.
    \end{proof}

    \begin{cor}\label{cor:TL-compute}
        For any basis element $\tau$ of $TL_n(2)$, we have
        \[ \Imm^{\TL}_\tau(A_{\lambda/\mu,R}) = \sum_{\substack{\text{$T$ of shape $(\lambda_R/\mu_R)\circledast(\lambda_B/\mu_B)$} \\ \psi(T) = \tau}} \omega(T). \]
    \end{cor}
    
\subsection{Crystal operators}

    Now we explain how crystal operators act on shuffle tableaux, as defined in \cite{nguyen2025temperley}.

    Given a shuffle tableau $T$, an $(i,i+1)$-overlap is a pair of squares $(s,t)$ such that $s$ contains an $i$, $t$ contains an $i+1$, and $s$ and $t$ are on the same column. If $(s,t)$ is an $(i,i+1)$-overlap pair, we say that $s$ and $t$ are $(i,i+1)$-overlapped. 
    
    \begin{example}\label{exp:overlap}
        In the following tableau $T$,
        \[ \begin{ytableau}    
        \none & \none & \none & \none & \textcolor{red}{1} & \none & \textcolor{red}{1} & \none & \textcolor{red}{1} & \none & 1 & \none & 2 \\ 
        \none & \none & \none & \textcolor{red}{1} & \none & \textcolor{blue}{2} & \none & 3 & \none & 3 & \none & 3 \\ 
        \none & \none & \textcolor{red}{1} & \none & \textcolor{red}{2} & \none & \textcolor{red}{2} & \none & \textcolor{red}{2} & \none & 3 \\ 
        \none & 2 & \none & \textcolor{red}{2} & \none & \textcolor{blue}{3} \\ 
        2 & \none & \textcolor{red}{2}
        \end{ytableau}, \]
        there are five pairs $(1,2)$-overlap (in red), and one pair of $(2,3)$-overlap (in blue).
    \end{example}

    Next, we define the $i$-reading word $w_i(T)$ as follows.
    \begin{enumerate}
        \item Consider the squares consisting of $i$ and $i+1$.
        \item Remove all $(i,i+1)$-overlap pairs.
        \item Iteratively read the remaining squares from bottom to top, left to right.
    \end{enumerate}

    \begin{example}\label{exp:reading-word}
        In the tableau in Example \ref{exp:overlap}, we have
        \[ w_1(T) = 2~|~2~|~|~2~|~1~2 \]
        by reading all non-red $1$s and $2$s, and
        \[ w_2(T) = 2~2~|~2~2~|~2~2~2~3~|~3~3~3~|~2 \]
        by reading all non-blue $2$s and $3$s.
    \end{example}

    \begin{definition}[{\cite[Definition 4.3]{nguyen2025temperley}}]\label{def:crys-ops}
        The \textit{crystal operators} $E_i$ and $F_i$ act on $w_i(T)$ in the following (standard) way. 
        \begin{itemize}
            \item View each $i$ as a closing parenthesis ``)'' and each $i+1$ as an opening parenthesis ``(''.
            \item Match the parentheses in the usual way.
            \item $E_i$ changes the leftmost unmatched $i+1$ to $i$, and $F_i$ changes the rightmost unmatched $i$ to $i+1$.
        \end{itemize}

        We induce this action of the crystal operators to the action on shuffle tableaux as follows: 
        \begin{itemize}
            \item the shape of the shuffle tableaux is preserved;
            \item The content changes in the way uniquely determined by $w_i(E_i(T)) = E_i(w_i(T))$ and $w_i(F_i(T)) = F_i(w_i(T))$.
        \end{itemize} 

        A \textit{Temperley--Lieb crystal} is a connected component of the graph on shuffle tableaux formed by the crystal operators.
   \end{definition}

   It was proved in \cite{nguyen2025temperley} that these cyrstal operators $E_i, F_i$ satisfy Stembridge's axioms and do not change the Temperley--Lieb type.

   \begin{prop}[{\cite[Proposition 4.7]{nguyen2025temperley}}]\label{prop:type-preserve}
       $E_i$ and $F_i$ do not change the Temperley--Lieb type. That is,
        \[ \psi(T) = \psi(E_iT) = \psi(F_iT). \]
   \end{prop}

   \begin{thm}[{\cite[Theorem 5.1]{nguyen2025temperley}}]\label{thm:type-A}
        Temperley--Lieb crystals are type A Kashiwara crystals.
    \end{thm}

    Combining Corollary \ref{cor:TL-compute}, Proposition \ref{prop:type-preserve}, and Theorem \ref{thm:type-A}, we have Theorem \ref{thm:main-thm}.

    \begin{proof}[Proof of Theorem \ref{thm:main-thm}]
        Corollary \ref{cor:TL-compute} means that one can compute $\Imm_\tau^{\TL}(A_{\lambda/\mu,R})$ by summing over shuffle tableaux of the corresponding shape and Temperley--Lieb type $\tau$. Theorem \ref{thm:type-A} means that we can split the set of all shuffle tableaux into Temperley--Lieb crystals, and Proposition \ref{prop:type-preserve} implies that every shuffle tableau in each crystal has the same Temperley--Lieb type. Thus, each Temperley--Lieb immanant is a sum of some Temperley--Lieb crystals and hence is Schur positive.
    \end{proof}

    \begin{remark}
        To obtain the Littlewood--Richardson coefficients, i.e. $\langle \Imm^{\TL}_\tau(A_{\lambda/\mu,R}), s_\nu \rangle$, one only needs to count \textit{Yamanouchi shuffle tableaux}, which are shuffle tableaux on which no crystal operator $E_i$ can be applied. An efficient way to generate Yamanouchi shuffle tableaux is via peelable tableaux in \cite{nguyen2025shuffle}.
    \end{remark}

\section*{Acknowledgement}

    We thank Nhi Dang for explaining to us the details about Hamel--Goulden matrices and creating the software to generate them. After the completion of this work, we were informed that Angele Foley, Alejandro Morales, and Daniel Soskin had independently obtained the same result. This work was supported by a grant from Simons Foundation International SFI-MPS-SFM-00011393 P.P.

\bibliography{bibliography}
\bibliographystyle{alpha}

\end{document}